\documentclass{amsart}
\usepackage{enumerate}
\usepackage{tikz}
\usetikzlibrary{automata,positioning}
\usepackage{hyperref}
\usepackage{mathrsfs}
\usepackage{amsfonts}
\usepackage{latexsym}
\usepackage{epsfig}
\usepackage{indentfirst}
\graphicspath{{figures/}}
\usepackage{amsmath}
\usepackage{amsthm}
\usepackage{graphicx}
\usepackage{amssymb}
\usepackage{amscd}
\usepackage{latexsym}
\usepackage{subfigure}
\usepackage{epstopdf}
\usepackage{color}
\DeclareGraphicsExtensions{.eps}
\input xy
\xyoption{all}

\DeclareMathOperator{\coker}{Coker}

\newtheorem{thm}{Theorem}[section]
\newtheorem{lem}[thm]{Lemma}
\newtheorem{cor}[thm]{Corollary}
\newtheorem{prop}[thm]{Proposition}

\theoremstyle{definition}
\newtheorem{defn}[thm]{Definition}

\newtheorem{que}[thm]{Question}

\theoremstyle{remark}
\newtheorem{rmk}[thm]{Remark}

\numberwithin{equation}{section}

\numberwithin{equation}{section} \makeatletter

\begin{document}

\title{A lower bound for the double slice genus}

\author{Wenzhao Chen}
\address{Max Planck Institute of Mathematics, Vivatsgasse 7, 53111 Bonn, Germany}
\email{chenwenz@msu.edu}
\thanks{}



\begin{abstract}
In this paper, we develop a lower bound for the double slice genus of a knot using Casson-Gordon invariants. As an application, we show that the double slice genus can be arbitrarily larger than twice the slice genus. As an analogue to the double slice genus, we also define the superslice genus of a knot, and give both an upper bound and a lower bound in the topological category.
\end{abstract}
\maketitle

\section{Introduction}
A surface smoothly embedded in the 4-sphere is \textit{unknotted} if it bounds a handlebody. Regarding the 3-sphere $S^3$ as the equator of the 4-sphere, a knot $K\subset S^3$ is \textit{smoothly doubly slice} if it is the intersection of an unknotted 2-sphere with the equator $S^3$. Obviously not every knot is doubly slice, for there exists knots which are not even \textit{slice}, i.e. bounding properly embedded disks in the 4-ball. Moreover, not every slice knot is doubly slice. In fact, about six decades ago Fox posed a challenging question: determine which slice knots are doubly slice (cf.\ Problem 39 of \cite{Fox62}). Since then this question has been the center of the study of double sliceness, and many obstructions to double sliceness were found (e.g.\ \cite{Ste04, GL83, Kim06, LM15, Mei15, Ors17, Sum71}). Recently Livingston and Meier introduced a notion called the double slice genus of a knot, which sets this topic in larger context \cite{LM15}. We recall the definition below. 
\begin{defn}
Given a knot $K\subset S^3$, its double slice genus is defined as 
\begin{displaymath}
g_{ds}(K)=\min_{\mathcal{S}} \{ g(\mathcal{S})\\  |\ \mathcal{S}\ \text{is an unknotted surface in}\ S^4,\  \mathcal{S}\cap S^3 = K\},
\end{displaymath}
where we view $S^3$ as the equator of $S^4$. 
\end{defn}

Note $g_{ds}(K)$ is defined, for a surface $\mathcal{S}$ satisfying the above requirements exists for every $K$. To see this, let $F$ be a surface obtained by pushing the interior of some Seifert surface of $K$ into the 4-ball. Then the double of $F$ clearly bounds a 3-manifold homeomorphic to $F\times I$. Furthermore, this observation also implies the double slice genus is bounded above by twice the Seifert genus. On the other hand, it is straightforward to see $g_{ds}(K)$ is bounded below by twice the slice genus of $K$. In summary, we have
\begin{displaymath}
2g_4(K) \leq g_{ds}(K) \leq 2g_3(K).
 \end{displaymath}

Along the lines of Fox's question to tell sliceness and double sliceness apart, a natural question in this context is: can the double slice genus be arbitrarily larger than twice the slice genus? Answering this question requires a lower bound for the double slice genus. Note while many knot invariants give lower bounds for $g_4$, there previously are no algebraic invariants that improve on the lower bound $2g_4$ for $g_{ds}$. By using Casson-Gordon invariants of the two-fold branched cover in conjunction with another algebraic invariant that we define, we develop the first such lower bound in this paper. As a primary application, we prove
\begin{thm}\label{thm1}
There exist ribbon knots $K_n$, $n\in \mathbb{N}$, such that $$\lim_{n\rightarrow \infty} g_{ds}(K_n)=\infty.$$
\end{thm}
Closely related to the double sliceness is a notion called supersliceness. Recall that a knot $K$ is called \textit{superslice} if there is a slice disk $D$ whose double along $K$ produces an unknotted 2-sphere in $S^4$. As an analogue to the double slice genus, we define the superslice genus of a knot.
\begin{defn}
Given a knot $K\subset S^3$, its superslice genus is defined as 
\begin{displaymath}
\begin{aligned}
g^{s}(K)=\min_{F} \{ g(F)\,  |  \,F\,\text{is properly smoothly embedded in}\,D^4, \partial F= K\,\\
\text{and the double of\,}F\,\text{bounds a handlebody in\,}S^4 \}.
\end{aligned}
\end{displaymath}
\end{defn}

It is easy to see $g_{ds}(K)\leq 2g^s(K)$, and hence the lower bounds for the double slice genus hold for the superslice genus as well. However, greater rigidity encoded in the definition of the superslice genus compared to that of the double slice genus allows us to obtain a much more accessible bound. 

\begin{thm}\label{superslice lower bound}
Let $K$ be a knot in $S^3$ and $\Sigma$ be the two-fold branched cover of $S^3$ along $K$. Then the minimum number of generators of $H_1(\Sigma; \mathbb{Z})$ is a lower bound for $2 g^s(K)$.
\end{thm}

In fact, the idea contained in the proof of Theorem \ref{superslice lower bound} may serve as a prototype for lower bounds for the double slice genus. Compare Subsection 2.1 and Subsection 3.1. 

In addition to lower bounds, it is also natural to ask if one can give upper bounds for $g^s$ or $g_{ds}$. In this paper, we pursue this direction in the topological category, i.e. the surfaces used in Definition 1.1 and Definition 1.3 are allowed to be topologically embedded and locally flat, and denote the corresponding quantities by $g^s_{top}(K)$ and $g_{ds}^{top}(K)$. We remark that the topological category and the smooth category are different \cite{Mei15, Rub16}, and the lower bounds constructed in this paper also hold in the topological category. 

We give upper bounds in terms of the Alexander polynomial. Freedman proved knots with trivial Alexander polynomial are topologically slice \cite{MR804721, MR1201584}. Recently, Feller generalized this theorem: the degree of the Alexander polynomial is an upper bound for twice the topological slice genus \cite{Fel16}. Here the degree of the Alexander polynomial is the breadth of the polynomial. In the context of superslice genus, two results of Freedman imply that knots with trivial Alexander polynomial are exactly the topologically superslice knots; see \cite{LM15} and \cite{Mei15}, or see the discussion in Subsection 3.2. So it is natural to wonder if one can bound the topological superslice genus by the degree of the Alexander polynomial. Indeed, we have the following result. 

\begin{thm}\label{superslcie upper bound}
The degree of the Alexander polynomial of a knot is an upper bound for twice its topological superslice genus.
\end{thm}

This theorem has the following immediate corollary.
\begin{cor}\label{superslice genus for knots with alex of deg 1}
If the degree of the Alexander polynomial of a knot $K$ is $2$, then $g^s_{top}(K)=1$.
\end{cor}

\begin{rmk}
Similar inequalities as in Theorem \ref{superslice lower bound} and Theorem \ref{superslcie upper bound} also appeared in the context of \textit{$\mathbb{Z}$-slice genus} of a knot, i.e.\ minimal genus of surfaces in the 4-ball whose boundary is the given knot and whose complement has fundamental group isomorphic to $\mathbb{Z}$ (cf.\ Theorem 1 and Proposition 12 of \cite{FL18}). In fact, the proof of Thoerem \ref{superslcie upper bound} together with Theorem 1.1 of \cite{FL19} imply the topological superslice genus is equal to the $\mathbb{Z}$-slice genus. This was pointed out to the author by Peter Feller and Lukas Lewark. We refer the interested readers to their papers \cite{FL18, FL19} for a more systematic study of the $\mathbb{Z}$-slice genus.
\end{rmk}
The rest of the paper is organized as follows: the lower bounds for the double slice genus are constructed in Subsection 2.1--2.3. Theorem \ref{thm1} is proved in Subsection 2.4, modulo a technical lemma which is proved in the appendix. Theorem \ref{superslice lower bound} is proved in Subsection 3.1. Theorem \ref{superslcie upper bound} and Corollary \ref{superslice genus for knots with alex of deg 1} are proved in Subsection 3.2. 

\noindent\textbf{Acknowledgment:} I thank Peter Feller, Matt Hedden, Kristen Hendricks, Effie Kalfagianni and Chuck Livingston for their help and interest. Especially, Matt Hedden sacrificed his personal time discussing superslice genus with me during his very busy days, and Kristen Hendricks kindly provided extremely detailed feedback after reading the first draft. I am also indebted to the referee for many helpful suggestions. The author is grateful to the Max Planck Institute of Mathematics in Bonn for its hospitality and financial support. 

\section{Construction of the lower bound}
In this section, we construct a lower bound for the double slice genus and use it prove Theorem \ref{thm1}. This bound comes from studying $\Sigma(K)$, the two-fold branched cover of $S^3$ along a knot $K$ (we expect similar bounds can be defined using higher order branched covers): we first derive a lower bound in terms of the singular homology of $\Sigma(K)$ in Subsection 2.1, and then study the Casson-Gordon invariants of $\Sigma(K)$ to give other lower bounds in Subsection 2.2, finally we combine the bounds in Subsection 2.1-2.2 to give the desired lower bound in Subsection 2.3. In Subsection 2.4, we prove Theorem \ref{thm1}. 

\subsection{Double slice genus and the singular homology of $\Sigma(K)$}
Regard $S^3$ as the equator of $S^4=D^4_1\cup_{S^3} D^4_2$, where the spaces $D_i^4$'s are two copies of the 4-ball. Let $K$ be a knot in $S^3$ and let $F_{i}$ be properly embedded surfaces in $D^4_{i}$ such that $\partial F_{i}=K$, $i=1,2$, and $F=F_1\cup F_2$ bounds a handlebody in $S^4$. Let $g_i$ denote the genus of $F_i$, let $g=g_1+g_2$, and let $W_i$ be the two-fold branched cover of $D_i^4$ along $F_i$, $i=1,2$. Then $\partial W_1=-\partial W_2=\Sigma(K)$, and $W=W_1\cup_{\Sigma(K)} W_2$ is diffeomorphic to $\#_{g}S^2\times S^2$, the connected sum of $g$ copies of $S^2\times S^2$. In the set up, the genus of $F$ is captured by the homology groups of the 4-manifolds. Indeed, we have the fact that $b_2(W_i)=2g_i$ (see Proposition \ref{extendability} (i) for a proof). 

We use the homology groups of $\Sigma(K)$ to give lower bounds for $b_2(W_i)$ by examining various long exact sequences relating these spaces. For convenience, we write $\Sigma$ for $\Sigma(K)$ hereafter. Throughout this paper we use integer coefficient for the singular homology groups unless otherwise specified. Note $\Sigma$ is a rational homology sphere. 
From the Mayer-Vietoris sequence for $W_1\cup_{\Sigma} W_2$, we have
\begin{equation}\label{longexactsequence1}
0\rightarrow H_2(W_1)\oplus H_2(W_2)\rightarrow H_2(W) \rightarrow H_1(\Sigma) \rightarrow H_1(W_1)\oplus H_1(W_2) \rightarrow 0.
\end{equation}

The long exact sequence for the pair $(W, \Sigma)$ gives rise to
\begin{equation}\label{longexactsequence2}
0\rightarrow H_2(W) \rightarrow H_2(W,\Sigma) \rightarrow H_1(\Sigma) \rightarrow 0.
\end{equation}

From the long exact sequence for $(W_i, \Sigma)$, $i=1,2$, we obtain
\begin{equation}\label{longexactsequence3}
0\rightarrow  H_2(W_i) \rightarrow H_2(W_i,\Sigma) \rightarrow H_1(\Sigma) \rightarrow H_1(W_i) \rightarrow 0,
\end{equation}
where surjectivity of the map $H_1(\Sigma) \rightarrow H_1(W_i)$ is derived from (\ref{longexactsequence1}).

Finally, the long exact sequence for $(W, W_i)$, $i=1,2$ shows
\begin{equation}\label{longexactsequence4}
0\rightarrow  H_2(W_i) \rightarrow H_2(W) \rightarrow H_2(W,W_i) \rightarrow H_1(W_i) \rightarrow 0.
\end{equation}

Correspondingly, we can deduce the following properties of the homology groups involved in the above exact sequences.
\begin{prop}\label{extendability}
In the notation established above, we have the following properties.
\begin{enumerate}[(i)]
\item $H_2(W_i)\cong\mathbb{Z}^{2g_i}$, $i=1,2$.\\
\item $|H_1(W_i)|^2 \big| |H_1(\Sigma)|$, $i=1,2$.\\
\item There exists $a^i_j \in H_1(W_i)\oplus\mathbb{Z}^{2g_i}$ for $i=1,2$ and $j=1,2,...,2g$ such that
		\begin{enumerate}
			\item $\frac{ H_1(W_i)\oplus \mathbb{Z}^{2g_i}}{ \langle \{a^i_j |\ j=1,...,2g \} \rangle }\cong H_1(W_{i+1})$ for $i=1,2$, where we let $W_3=W_1$.
			\item $\frac{ (H_1(W_1)\oplus\mathbb{Z}^{2g_1})\oplus (H_1(W_2)\oplus\mathbb{Z}^{2g_2} )}								{\langle \{ (a^1_j,a^2_j)|\ j=1,...,2g\} \rangle}\cong H_1(\Sigma)$.
		\end{enumerate}
		Here $\langle \{a^i_j |\ j=1,...,2g \} \rangle$  stands for the subgroup generated by the set $\{a^i_j |\ j=1,...,2g \}$.

\end{enumerate}
\end{prop}

\begin{proof}
\begin{enumerate}[(i)]
\item The long exact sequence (\ref{longexactsequence1}) implies $H_2(W_i)$ is free abelian, since it is mapped injectively into $H_2(W)\cong\mathbb{Z}^{2g}$. The statement would then follow from the claim that $b_2(W_i)=2g_i$. To see this claim, note $\chi(W_i)=2\chi(D^4- (F_i\times D^2))+\chi(F_i\times D^2)-\chi(F_i\times S^1)=1+2g_i$, and we also have $\chi(W_i)=1-b_1(W_i)+b_2(W_i)-b_3(W_i)$. Long exact sequence (\ref{longexactsequence3}) implies $b_1(W_i)=0$. As we have $H_1(W_i)\rightarrow H_1(W_i,\Sigma)\rightarrow 0$ from the long exact sequence for the pair $(W_i, \Sigma)$, we have $b_3(W_i)=\text{rk}H^3(W_i)=\text{rk}H_1(W_i,\Sigma)\leq \text{rk} H_1(W_i)=b_1(W_i)=0$. Therefore, $b_2(W_i)=\chi(W_i)-1=2g_i$.

\item Consider the long exact sequence (\ref{longexactsequence3}).	First we claim $| \coker \{H_2(W_i) \rightarrow H_2(W_i,\Sigma)\} |$ is divisible by $|H_1(W_i)|$. To see the claim, note $H_2(W_i,\Sigma)\cong H^2(W_i)\cong H_2(W_i)\oplus H_1(W_i)$ by Poinc\'{a}re duality and the universal coefficient theorem. Assume $H_1(W_i)=\mathbb{Z}_{m_1}\oplus \cdots\oplus \mathbb{Z}_{m_k}$. Then $H_2(W_i)\oplus H_1(W_i)$ has a presentation matrix of the form 

$$\begin{bmatrix}
    m_1 & 0 & 0 & 0 & \cdots & 0  \\
    \vdots  & \ddots & \vdots & \vdots & \ddots &\vdots\\
       0    & 0  & m_k  & 0 & \cdots & 0
\end{bmatrix}_{\textstyle \raisebox{2pt}{.}}$$

Then $\coker \{H_2(W_i) \rightarrow H_2(W_i,\Sigma)\}$ has a presentation matrix 
$$\begin{bmatrix}
    m_1 & 0 & 0 & 0 & \cdots & 0  \\
    \vdots  & \ddots & \vdots & \vdots & \ddots &\vdots\\
       0    & 0  & m_k  & 0 & \cdots & 0\\
        * & * & * & * & \cdots & *  \\
    \vdots  & \ddots & \vdots & \vdots & \ddots &\vdots\\
       *   & *  & *  & * & \cdots & *
\end{bmatrix}_{\textstyle \raisebox{2pt}{.}}$$
Note $|\coker \{H_2(W_i) \rightarrow H_2(W_i,\Sigma)\}|$ is equal to the absolute value of the determinant of the above matrix, and hence is divisible by $m_1m_2 \cdots m_k=|H_1(W_i)|$. With this claim in mind, the statement then follows easily from (\ref{longexactsequence3}).

\item Note there are isomorphisms $$H_2(W,W_i)\cong H_2(W_{i+1},\Sigma)\cong H_2(W_{i+1})\oplus H_1(W_{i+1}) \cong H_1(W_{i+1})\oplus \mathbb{Z}^{2g_{i+1}}.$$
Here the first isomorphism follows from excision, and the second isomorphism is obtained by Poincar\'e duality and the universal coefficient theorem.  Now (a) follows from this and long exact sequence (\ref{longexactsequence4}). 

Identify a regular neighborhood of $\Sigma$ in $W$ as $\Sigma\times [-1,1]$. Let $W_1'= W_1-\Sigma\times (-1,0]$ and let $W_2'=W_2-\Sigma\times [0,1)$, which are obviously deformation retracts of $W_1$ and $W_2$ respectively. Then
\begin{displaymath}
\begin{aligned}
H_2(W, \Sigma)&\cong H_2(W,\Sigma \times [-1,1]) \\ &\cong H_2(W_1'\cup W_2',\Sigma\times\{0\}\cup\Sigma\times\{1\})\cong H_2(W_1,\Sigma)\oplus H_2(W_2,\Sigma). 
\end{aligned}
\end{displaymath}
Here the second isomorphism is obtained by excision. With this in mind, we have
$$H_2(W,\Sigma)\cong H_2(W_1,\Sigma)\oplus H_2(W_2,\Sigma)\cong H_1(W_1)\oplus \mathbb{Z}^{2g_1} \oplus H_1(W_2)\oplus \mathbb{Z}^{2g_2}.$$
Then (b) follows from this observation and long exact sequences (\ref{longexactsequence2}).

\end{enumerate} 
\end{proof}

These observations motivate the following definitions.
\begin{defn}\label{definition1}

Let $G$ be a finite abelian group. A pair of finite abelian groups $(G_1,G_2)$ is said to be admissible for $G$ if
	\begin{enumerate}
		\item[(i)] $|G_i|^2 \big| |G|$, $i=1,2$.
		\item[(ii)] There exist $n_1, n_2 \in \mathbb{N}$ and $a^i_j \in G_i\oplus\mathbb{Z}^{2n_i}$ for $i=1,2$ and $j=1,2,...,2(n_1+n_2)$, such that the following requirements are satisfied.
		\begin{itemize}
		\item[(1)] $\frac{G_i\oplus \mathbb{Z}^{2n_i}}{\langle  \{ a^i_j|, j=1,...,2(n_1+n_2)\}\rangle}\cong G_{i+1}$ for $i=1,2$, where we let $G_3=G_1$.
		\item[(2)]$\frac{(G_1\oplus \mathbb{Z}^{2n_1})\oplus(G_2\oplus \mathbb{Z}^{2n_2})}{\langle  \{ (a^1_j,a^2_j)| j=1,...,2(n_1+n_2)\}  \rangle}\cong G$.
		\end{itemize}
	\end{enumerate}

\end{defn}

\begin{defn}
Let $G$ be a finite abelian group and let $(G_1,G_2)$ be an admissible pair for $G$. We define a numerical invariant
$$\theta_1(G,G_1,G_2)=\min \{n_1+n_2 \mid n_1,\,  n_2 \,\text{are as in Definition \ref{definition1}} \}.$$
\end{defn}

It follows from Proposition \ref{extendability} that $$\theta_1(H_1(\Sigma),H_1(W_1),H_1(W_2))\leq g_{ds}(K).$$ However, $\theta_1$ has an apparent drawback in that $H_1(W_i)$ cannot be inferred from the knot $K$, and taking a further minimum over all possible admissible pairs would lead to rather trivial bounds. 

\subsection{Double slice genus and Casson-Gordon invariants}
In this subsection, we use Casson-Gordon invariants to construct a lower bound for the double slice genus. First we recall the relevant facts about Casson-Gordon invariants below. Detailed information may be found in \cite{CG86, Gil81, Rub83}.

We begin with recalling the definition of Casson-Gordon invariants. Let $M$ be an oriented three manifold equipped with a character $\phi: H_1(M)\rightarrow \mathbb{Z}_d$, where $d$ is a non-negative integer. By bordism theory, there exists some positive integer $r$ such that $r\cdot (M,\phi)=\partial (V,\phi')$, where $V$ is a compact 4-manifold and $\phi': H_1(V)\rightarrow \mathbb{Z}_d$ is a character that restricts to $\phi$ on the boundary. $\phi'$ determines a cyclic cover $\widetilde{V}\rightarrow V$ with a preferred covering transformation $T \colon \widetilde{V}\rightarrow \widetilde{V}$. Let $T_* \colon H_2(\widetilde{V};\mathbb{C})\rightarrow H_2(\widetilde{V};\mathbb{C})$ be the induced automorphism and let $\bar{H}_2(V,\phi')$ be the $e^{2\pi i/d}-\text{eigenspace of }T_*$. Note that the intersection form on $H_2(\widetilde{V};\mathbb{Z})$ extends naturally to a Hermitian pairing  $\langle\, , \, \rangle$ on $H_2(\widetilde{V};\mathbb{C})=H_2(\widetilde{V};\mathbb{Z})\otimes \mathbb{C}$. Let $\bar{\sigma}(V,\phi')$ denote the signature of this Hermitian pairing restricted to $\bar{H}_2(V,\phi')$. Then define the Casson-Gordon invariant associated to $(M,\phi)$ as $$\sigma(M,\phi)=\frac{1}{r}(\bar{\sigma}(V,\phi')-\sigma(V)).$$ Here $\sigma(V)$ denotes the usual signature invariant of $V$. 

The key fact that relates Casson-Gordon invariants and the double slice genus is the following proposition due to Gilmer. 
\begin{prop}[Proposition 1.4 of \cite{Gil81}]\label{smiththeory}
If $\phi'$ is of prime power order, i.e. $d=p^n$ for some prime $p$ and positive integer $n$, let $\bar{b}_2(V):=\dim_\mathbb{C} \bar{H}_2(V,\phi')$ and let $b_2(V,\mathbb{Z}_p):=\dim_{\mathbb{Z}_p} H_2(V;\mathbb{Z}_p)$. Then $$ \bar{b}_2(V) \leq b_2(V,\mathbb{Z}_p).$$
\end{prop}

With the above preparation on Casson-Gordon invariants, we return to the double slice genus. Let $K$, $W_i$ for $i=1,2$, $W$ and $\Sigma$ be as in Subsection 2.1. Let $p$ be a prime, and for a finite abelian group $G$, define $\xi_p(G):=\dim_{\mathbb{Z}_p}G\otimes \mathbb{Z}_p$. 
\begin{thm}\label{lowerbound by CG} 
Let $d=p^n$ and $\phi_i: H_1(\Sigma)\rightarrow \mathbb{Z}_d$ be a character that factors through $H_1(W_i)$ along the inclusion induced homomorphism $\iota_i: H_1(\Sigma)\rightarrow H_1(W_i)$, $i=1,2$. Then
\begin{enumerate}[(i)]
\item $|\sigma(\Sigma,\phi_1)-\sigma(\Sigma,\phi_2)|-\xi_p(H_1(W_1)\oplus H_1(W_2))\leq 2g_{ds}(K).$

\item $|\sigma(\Sigma,\phi_i)+\sigma(K)|-\xi_p(H_1(W_i))\leq 2g_i.$
\end{enumerate}  
\end{thm}

\begin{proof}
\leavevmode
\begin{enumerate}[(i)]
\item  Note $\Sigma=\partial W_1=-\partial W_2$, hence 
\begin{displaymath}
\begin{aligned}
|\sigma(\Sigma,\phi_1)-\sigma(\Sigma,\phi_2)|&=|\bar{\sigma}(W_1,\phi_1)+\bar{\sigma}(W_2,\phi_2)-\sigma(W_1)-\sigma(W_2)|\\
&=|\bar{\sigma}(W_1,\phi_1)+\bar{\sigma}(W_2,\phi_2)-\sigma(W)|\\
&=|\bar{\sigma}(W_1,\phi_1)+\bar{\sigma}(W_2,\phi_2)|\\
&\leq \bar{b}_2(W_1)+\bar{b}_2(W_2)\\
&\leq b_2(W_1;\mathbb{Z}_p)+b_2(W_2;\mathbb{Z}_p)\\
&=2g_1+\xi_p(H_1(W_1))+2g_2+\xi_p(H_1(W_2)).
\end{aligned}
\end{displaymath}
Here we used Novikov additivity for the second equality, $\sigma(W)=\sigma(\#_g S^2\times S^2)=0$ for the third equality, and the universal coefficient theorem for the last equality.  Finally, note $g_{ds}(K)=g_1+g_2$ and $\xi_p(H_1(W_1)\oplus H_1(W_2))=\xi_p(H_1(W_1))+\xi_p(H_1(W_2))$, hence the statement follows.

\item Note that $\sigma(K)=\sigma(W_1)=-\sigma(W_2)$ by Theorem 3.1 of \cite{KT76}. Similar to the argument in (i) above we have
\begin{displaymath}
\begin{aligned}
|\sigma(\Sigma,\phi_i)+\sigma(K)|&=|\bar{\sigma}(\Sigma,\phi_i)|\\
&\leq \bar{b}_2(W_i)\\
&\leq b_2(W_i;\mathbb{Z}_p)\\
&=2g_i+\xi_p(H_1(W_i)).
\end{aligned}
\end{displaymath}
The statement readily follows.
  \end{enumerate}
\end{proof}

The following definition is motivated by the above theorem. 
\begin{defn}
Let $(G_1,G_2)$ be an admissible pair for $H_1(\Sigma)$. Define
$$
\begin{aligned}
\theta_2(\Sigma,G_1,G_2)=&\frac{1}{2} \min_{(\iota_1, \iota_2)}\max_{(\phi_1,\phi_2,p)}\{|\sigma(\Sigma,\phi_1)-\sigma(\Sigma,\phi_2)|-\xi_p(G_1\oplus G_2)\, \big| \, \\ &\phi_i\, \text{factors through}\, G_i \, \text{via}\, \iota_i\, \text{and its order is a power of the prime}\, p\}
\end{aligned}
$$
and
$$
\begin{aligned}
\theta_3(K,G_1,G_2)=&\frac{1}{2}\min_{(\iota_1, \iota_2)}\max_{(\phi_1,\phi_2,p)}\{\max(0, |\sigma(\Sigma,\phi_1)+\sigma(K)|-\xi_p(G_1))+\\ &\max(0,|\sigma(\Sigma,\phi_2)+\sigma(K)|-\xi_p(G_2))\,\big| \, \\ &\phi_i\, \text{factors through}\, G_i \, \text{via}\, \iota_i\, \text{and its order is a power of the prime}\, p\},
\end{aligned}
$$
where in both equations the maximum is taken over all primes $p$ and characters $\phi_1$ and $\phi_2$ satisfying the constraints, and the minimum is taken over all surjective homomorphisms $\iota_i: H_1(\Sigma)\rightarrow G_i$ for $i=1,2$.
\end{defn}

In view of Theorem \ref{lowerbound by CG} we clearly have $\theta_2(\Sigma,H_1(W_1),H_1(W_2))\leq g_{ds}(K)$ and $ \theta_3(K,H_1(W_1),H_1(W_2))\leq g_{ds}(K)$. However, like $\theta_1$, these invariants are difficult to utilize since one has little control of $H_1(W_i)$ for $i=1,2$.

\subsection{Combining $\theta_i$}
So far we have defined various $\theta_i$'s, all of which require the input of an admissible pair that cannot be deduced from the knot. One obvious remedy is to take a minimum over all the admissible pairs, which unfortunately does not lead to a useful lower bound if one uses a single $\theta_i$. However, this can be overcome by combining these invariants. First note that $\theta_1$ and $\theta_2$ are well defined if we replace $\Sigma$ with an arbitrary rational homology sphere. This allows us to make the following definition.

\begin{defn}
\begin{enumerate}[(i)]
\item Given a rational homology $3$-sphere $Y$, define
$$\delta(Y)=\min_{(G_1,G_2)}\max\{\theta_1(H_1(Y),G_1,G_2), \theta_2(Y,G_1,G_2)\}.$$
Here the minimum is taken over all admissible pairs $(G_1,G_2)$ for $H_1(Y)$.

\item Let $K$ be a knot in $S^3$, and let $\Sigma$ be the two-fold branched cover of $S^3$ along $K$. Define
$$\delta(K)=\delta(\Sigma).$$
\end{enumerate}
\end{defn}

Recall every closed, orientable $3$-manifold $Y$ embeds in $\#_n S^2\times S^2$ for sufficiently large $n$, and the minimum such $n$ is defined to be the \emph{embedding number} $\epsilon(Y)$ \cite{AGL17}. We have the following theorem on the $\delta$-invariant.
\begin{thm}
Let $Y$ be a rational homology 3-sphere. Then $$\delta(Y)\leq \epsilon(Y).$$ In particular, for a knot $K\subset S^3$, we have $\delta(K)\leq \epsilon(\Sigma) \leq g_{ds}(K)$.
\end{thm}

\begin{proof}
Embedding $Y$ in $\#_{\epsilon(Y)} S^2\times S^2$ separates $\#_{\epsilon(Y)} S^2\times S^2$ into two 4-manifolds $V_1$ and $V_2$. The proofs of Theorem \ref{extendability} and Theorem \ref{lowerbound by CG} carry over verbatim with $\Sigma$, $W_1$ and $W_2$ replaced by $Y$, $V_1$ and $V_2$. Therefore, 
\begin{displaymath}
\begin{aligned}
\delta(Y)&\leq \max \{\theta_1(H_1(Y),H_1(V_1),H_1(V_2)), \theta_2(Y, H_1(V_1), H_1(V_2))\}\\
&\leq \frac{1}{2}b_2(V_1\cup_Y V_2) = \frac{1}{2}b_2(\#_{\epsilon(Y)} S^2\times S^2)\\
&=\epsilon(Y).
\end{aligned}
\end{displaymath}
Given a knot $K$, note the two-fold branched cover $\Sigma$ embeds in $\#_{g_{ds}(K)} S^2 \times S^2$ (see the first paragraph of Subsection 2.1). Therefore, we have $\epsilon(\Sigma) \leq g_{ds}(K)$.
\end{proof}

When only interested in knots, we can use $\theta_3$ instead of $\theta_2$ to give a better lower bound for the double slice genus.

\begin{defn}\label{definition for theta}
Let $K$ be a knot in the 3-sphere. Define 
$$\theta(K)=\min_{(G_1,G_2)}\max\{\theta_1(H_1(\Sigma(K)),G_1,G_2), \theta_3(K,G_1,G_2)\}.$$
Here the minimum is taken over all admissible pairs $(G_1,G_2)$ for $H_1(\Sigma(K))$.
\end{defn}
\begin{thm}\label{theta is lower bound for ds gensu}
For any knot $K\subset S^3$, $\theta(K)\leq g_{ds}(K)$.
\end{thm}
\begin{proof}
Let $W_1$, $W_2$ and $W$ be as in Subsection 2.1. Then in view of Proposition \ref{extendability} and Theorem \ref{lowerbound by CG}, we have
\begin{displaymath}
\begin{aligned}
\theta(K)&\leq \max \{\theta_1(H_1(\Sigma(K)),H_1(W_1),H_1(W_2)), \theta_3(K,H_1(W_1),H_1(W_2))\}\\
&\leq \frac{1}{2}b_2(W_1)+\frac{1}{2}b_2 (W_2) = \frac{1}{2}b_2(W)\\
&=g_{ds}(K).
\end{aligned}
\end{displaymath}
\end{proof}

\subsection{Proof of Theorem \ref{thm1}}
In this subsection we prove Theorem \ref{thm1} using the $\theta$-invariant of Definition \ref{definition for theta}.

We begin by constructing the knots. Take $J$ to be the two-bridge knot corresponding to $\frac{9}{4}$, which is known to be ribbon (e.g. see \cite{CG86}). Let $K_n=\overbrace{J\#\cdots\#J}^{n}$. Note $\Sigma(J)=L(9,4)$ and hence $\Sigma(K_n)=\overbrace{L(9,4)\#\cdots \#L(9,4)}^{n}$. 

To estimate $\theta(K_n)$, we need to understand the behavior of Casson-Gordon invariants of $\Sigma(K_n)$. This
is addressed in the following technical proposition.

\begin{prop}\label{technicalproposition}
Let $m$ be a nonnegative integer and $s: H_1(\Sigma(K_n))\rightarrow \overbrace{\mathbb{Z}_9\oplus\cdots\oplus \mathbb{Z}_9}^{m}$ be a surjective map, then there exists a map $j:\overbrace{\mathbb{Z}_9\oplus\cdots\oplus \mathbb{Z}_9}^{m}\rightarrow \mathbb{Z}_9$ such that $\sigma(\Sigma(K_n), j \circ s)\geq \frac{10}{9}m$.
\end{prop}
 
 The proof of this proposition appears in Appendix A. 
 
Theorem \ref{thm1} follows from the next theorem. 
 \begin{thm}\label{small thm1}
 $\theta(K_{110n}) \geq n$, and hence $g_{ds}(K_{110n}) \geq n$.
 \end{thm}
 \begin{proof}
Let $(G_1, G_2)$ be an admissible pair for $H_1(\Sigma(K_{110n}))$. Since $G_1\oplus G_2$ is a quotient group of $H_1(\Sigma(K_{110n}))=\oplus_{110n}\mathbb{Z}_9$, we may write $G_1\oplus G_2=\overbrace{\mathbb{Z}_9\oplus\cdots\oplus\mathbb{Z}_9}^{l}\oplus\mathbb{Z}_3\oplus\cdots\oplus\mathbb{Z}_3$ for some $l$. We will prove $\theta(K_{110n})\geq n$ by considering the possible values of $l$ in two cases.
 
First, if $110n-l\geq 2n$, then we must have $\theta_1(H_1(\Sigma),G_1,G_2)\geq n$ in view of (ii)-(2) of Definition \ref{definition1}, since one must employ at least another $110n-l$ generators of order $9$ to get to $H_1(\Sigma)$. 

Second, if $110n-l<2n$, then $l>108n$. Then at least one of $G_i$, say $G_1$, has $l'$ many $\mathbb{Z}_9$ summands with $l'\geq 54n$. By Proposition \ref{technicalproposition} we have a character $\phi_1:H_1(\Sigma(K_{110n}))\rightarrow \mathbb{Z}_9$ that factors through $\iota_1: H_1(\Sigma(K_{110n}))\rightarrow G_1$, such that $\sigma(\Sigma(K_{110n}),\phi_1)\geq \frac{10}{9}l' .$  Also note that $\xi_3(G_1)\leq l'+110n-l$ and $\sigma(K_{110n})=0$. Therefore,
$$
\begin{aligned}
\theta_3(K_{110n},G_1,G_2) \geq& \frac{1}{2}(\frac{10}{9}l'-(l'+110n-l))\\
\geq& \frac{1}{2}(\frac{l'}{9}-2n)\\
\geq& \frac{1}{2}(\frac{54n}{9}-2n)\\
\geq & n.
\end{aligned}
$$

Therefore, $\theta(K)\geq n$ in either case.
\end{proof}
\begin{rmk}
For any given nonnegative integer $m$, one can similarly prove there is a family of knots whose slice genera are all equal to $m$ and double slice genera grow arbitrarily large. In fact, taking the connected sum of $m$ copies of the trefoil knot and $K_n$ produces such examples.
\end{rmk}
However, the author is not able to prove the embedding numbers of $\Sigma(K_n)$ grow arbitrarily large by the $\delta$-invariant defined in Subsection 2.3. The lower bounds for the embedding number in \cite{AGL17} come from lower bounds for the 2-nd betti number of spin 4-manifolds bounded by the given 3-manifold, and cannot be applied here since our 3-manifolds already bound rational homology balls. It seems natural to ask the following question. 
\begin{que}
Can one find a family of rational homology spheres that are boundaries of spin rational homology balls and whose embedding numbers can grow arbitrarily large? 
\end{que}
Moreover, the author wonders if one can find examples so that the $\delta$-invariant can be applied to answer the above question.
\section{Bounds for the superslice genus}
\subsection{A lower bound for the superslice genus}
We prove Theorem \ref{superslice lower bound} in this subsection. As a preparation we begin with two lemmas. 
\begin{lem}\label{lemma, trivial surface complement has cyclic fundamental group}
Let $F$ be a surface properly embedded in $D^4$ whose double is a closed surface that bounds a handlebody in $S^4$. Then $\pi_1(D^4-F)\cong \mathbb{Z}$. 
\end{lem}
\begin{proof}
Write the double of $F$  as $F_+\cup F_-$. Since $F_+\cup F_-$ bounds a handlebody in $S^4$, $\pi_1(S^4-(F_+\cup F_-))\cong \mathbb{Z}$. Applying Van Kampen's theorem, we have the following pushout diagram. Using the universal property we see there is a surjective map $\pi_1(S^4-(F_+\cup F_-))\cong \mathbb{Z}\rightarrow \pi_1(D^4-F)$. Therefore $\pi_1(D^4-F)$ is a cyclic group and must be isomorphic to $\mathbb{Z}$, since $H_1(D^4-F)\cong\mathbb{Z}$ by Alexander duality. 
\begin{center}
\begin{tikzpicture}[%
    >=stealth,
    shorten >=2pt,
    shorten <=2pt,
    auto,
    node distance=1.4cm
  ]
    \node (a) {$\pi_1(S^3-K)$};
    \node (c)  [node distance=1.4 cm, right of=a, above of=a] {$\pi_1(D^4-F_+)$};
    \node (b) [node distance=3.2 cm, right of=a] {$\pi_1(S^4-(F_+\cup F_-))$};
	\node (d) [node distance=1.4 cm, right of=a, below of=a] {$\pi_1(D^4-F_-)$};
	\node (f)	[node distance=4.5cm, right of=b]{$\pi_1(D^4-F)$};
   
     \path[->] (a) edge       (c);        
     \path[->] (a) edge        (d);   
     \path[->] (c) edge        (b); 
     \path[->] (d) edge       (b);   
     \path[->] (c) edge     node (i) {$Id$}   (f);     
     \path[->] (d) edge     node[swap] (ii) {$Id$}   (f);   
     \path[->,dashed] (b) edge     (f);

	 \end{tikzpicture}
\end{center}

\end{proof}
\begin{lem}
Let $F$ be as in the previous lemma, and let $W$ be the two-fold branched cover of $D^4$ along $F$. Then $H_1(W)=0$.
\end{lem}
\begin{proof}
Let $\tilde{F}\subset W$ be the lift of $F$. Then by Lemma 3.1, $\pi_1(W-\tilde{F})\cong \mathbb{Z}$.
Therefore $H_1(W-\tilde{F}) \cong \mathbb{Z}$, generated by the homology class of a meridian of $\tilde{F}$. Gluing $\tilde{F}$ back annihilates this and hence $H_1(W)=0$.
\end{proof}

\begin{proof}[Proof of Theorem \ref{superslice lower bound}]
Assume $F$ is a surface achieving the minimal superslice genus of the knot $K$, and $W$ is the two-fold branched cover of $D^4$ along $F$. Furthermore, let $\Sigma=\partial W$ denote the two-fold branched cover of $S^3$ along $K$. Note that $H_2(\Sigma)=0$, and $H_1(W)=0$ by the previous lemma. Then the long exact sequence associated to the pair $(W,\Sigma)$ gives 
$$
0\rightarrow H_2(W)\rightarrow H_2(W,\Sigma)\rightarrow H_1(\Sigma)\rightarrow 0.
$$

Note $H_2(W,\Sigma)\cong H^2(W)$, and $H^2(W)$ is free abelian by the universal coefficient theorem and the fact that $H_1(W)=0$. Since $b_2(W)=2g(F)$ by Proposition \ref{extendability} (i) and $g(F)=g^s(K)$ by our assumption, we have a presentation for $H_1(\Sigma)$:
$$
0\rightarrow \mathbb{Z}^{2g^s(K)}\rightarrow \mathbb{Z}^{2g^s(K)} \rightarrow H_1(\Sigma)\rightarrow 0.
$$

Hence the theorem follows.
\end{proof}

\subsection{An upper bound for the topological superslice genus}
In this subsection we prove Theorem \ref{superslcie upper bound}. We first recall three non-trivial results.

\begin{thm}[Theorem 7 of \cite{MR804721}, Theorem 11.7B of \cite{MR1201584}]\label{freedman disk theorem}
Let $K$ be a knot in $S^3$ such that $\Delta_K(t)=1$, then $K$ bounds a locally flat, topologically embedded disk $D\subset D^4$ such that $\pi_1(D^4-D)\cong \mathbb{Z}$.
\end{thm}

\begin{thm}[Theorem 6 of \cite{MR804721}, Theorem 11.7A of \cite{MR1201584}]\label{spherical unknotting theorem}
A locally flat embedding $f:S^2\rightarrow S^4$ is unknotted if and only if $\pi_1(S^4-f(S^2))\cong \mathbb{Z}$. 
\end{thm}

\begin{prop}[Proposition 2 of \cite{Fel16}]\label{feller's proposition}
Let $K$ be a knot. Every Seifert surface $F_K$ of $K$ contains a
simple closed curve $J$ separating $F_K$ into two subsurfaces $C_{K,J}$ and $F_J$ such that:
\begin{enumerate}[(i)]
\item The Alexander polynomial of $J$ is trivial.
\item $F_J$ is a Seifert surface for $J$ with
$2g(F_J) = 2g (F_K) -\text{deg}(\Delta_K(t))$.
\end{enumerate}
\end{prop}

We move on to apply these results to prove Theorem \ref{superslcie upper bound}. First, note Theorem \ref{freedman disk theorem} and Theorem \ref{spherical unknotting theorem} together implies
\begin{prop}[Theorem 4.5 of \cite{LM15}, Corollary 3.3 of \cite{Mei15}]\label{proposition, knots with trivial alexander polynomial is trivial}
Knots with trivial Alexander polynomial are topologically superslice.
\end{prop}
\begin{proof}
Doubling a slice disk asserted as in Theorem \ref{freedman disk theorem} produces a 2-knot satisfying the condition of Theorem \ref{spherical unknotting theorem}. 
\end{proof}

We set up some terminology for convenience. Given a superslice knot $J$, a 3-ball (locally flatly) embedded in $S^4$ is called a \emph{superslice ball} for $J$ if its boundary is the double of some superslice disk for $J$. Given a submanifold $N$ (possibly with non-empty boundary) of a manifold $M$, by a \emph{normal bundle} of $N$ we mean a disk bundle $E$ over $N$ together with an embedding of $E$ into $M$ such that $N$ is identified with the $0$-section of $E$. In this subection, the normal bundles we will encounter are always trivial. By abusing notation, we denote a trivial normal bundle $\phi:N \times D^k\rightarrow S^n$ of $N\subset S^n$ by $N\times D^k$ for appropriate $n$ and $k$, with $*\times D^k$ understood as $\phi(\{*\}\times D^k)$. Let $F$ an orientable surface (possibly with boundary) in $S^3$ and $\alpha: [0,1] \rightarrow S^3$ be embedded arc. We use the letter $\alpha$ to denote the map interchangeably with its image $\alpha([0,1])$ by abusing notation. Assume $\alpha \cap F=\{\alpha(0),\alpha(1) \}$ and $\alpha$ approaches both ends from the same side in a normal direction of $F$, we can perform a \emph{stabilization} of $F$ along $\alpha$: Take a normal bundle $\alpha\times D^2\subset S^3$ of $\alpha$ such that $(\alpha\times D^2) \cap F=\{\alpha(0),\alpha(1) \}\times D^2$, then a stabilization of $F$ along $\alpha$ is defined to be the surface $F'=(F-\{\alpha(0),\alpha(1) \}\times D^2 )\cup (\alpha \times \partial D^2)$.

Given a knot $K$, the idea for proving Theorem \ref{superslcie upper bound} is patching together a superslice ball for $J$ and a thickened Seifert surface for $K$, where $J$ is a knot as in Proposition \ref{feller's proposition}. To acheive this, one must understand the cross-section of a superslice ball at the equator $S^3$. The key lemma towards this goal is the following.

\begin{lem}\label{lemma, stabilizing the cross-section}
Let $B_J$ be a superslice ball for $J$. Assume $B_J$ intersects $S^3$ transversally and denote $B_J\cap S^3= S_0\cup S_1\cup \cdots \cup S_n$, where $S_0$ is a Seifert surface for $J$ and $S_1,\ldots, S_n$ are closed orientable surfaces. Let $S_i'$ be a stabilization of $S_i$ along some arc whose interior is disjoint from all the $S_i$'s. Then there is a boundary-fixing isotopy of $B_J$ to $B_J'$ such that $B_J'\cap S^3=S_0\cup \cdots \cup S_i'\cup\cdots \cup S_n \cup \partial(nb(L))$, where $L$ is a link in the complement of $S_0\cup \cdots \cup S_i'\cup\cdots \cup S_n$ and $nb(L)$ is a regular neighborhood of $L$.
\end{lem}

\begin{proof}
Denote the arc along which we stabilize $S_i$ by $\alpha$. We first pick an embedded arc $\beta$ in $B_J$ such that $\beta \cap (S_0\cup\cdots\cup S_n)=\{\alpha(0),\alpha(1)\}$. Roughly, $\beta$ is obtained by pushing the interior of a path in $S_i$ connecting $\alpha(0)$ and $\alpha(1)$ into $B_J\backslash S_i$. For the sake of a clear discussion, we fix a specific choice of $\beta$. First, parametrize a closed neighborhood of $S^3$ by $S^3\times[-1,1]$, where we identify the equator $S^3$ as $S^3\times\{0\}$ and require $S^3\times [-1,1]$ to be thin enough so that $B_J\cap (S^3\times [-1,1])=(S_0\cup S_1\cup \cdots \cup S_n)\times [-1,1]$. Let $\beta': [0,1]\rightarrow S_i$ be an embedded path connecting $\alpha(0)$ and $\alpha(1)$, and let $\phi: [0,1]\rightarrow [0,1]$ be the function $t\mapsto \sqrt{\frac{1}{4}-(t-\frac{1}{2})^2}+\frac{1}{3}$. We then set
\begin{displaymath}
\begin{aligned}
&\beta :\left[ 0,1\right] \rightarrow S_{i}\times \left[ 0,1\right] \\
 &t\mapsto \begin{cases}
 \left( \alpha \left( 0\right) ,t\right),\,\,\,\,\,\,\,\,\,\,\,\,\,\,\,\,\,\,\,\,\,\,\,\,\,\,\,\,\,\,\,\,\,\ \ \ \ \ 0\leq t\leq \frac {1}{3}\\
  \left( \beta '\left( 3t-1\right) ,\phi \left( 3t-1\right) \right),\,\,\,\,\,\,\frac {1}{3}\leq t\leq \frac {2}{3}\\
   \left( \alpha \left( 1\right) ,1-t\right),\,\,\,\,\,\,\,\,\,\,\,\,\,\,\,\,\,\,\,\,\,\,\,\,\,\,\,\,\,\,\,\,\,\frac {2}{3}\leq t\leq 1.
   \end{cases}
\end{aligned}
\end{displaymath} 

The idea for proving this lemma is to find an embedded disk $W$ bounded by $\alpha\cup \beta$ which only intersects $B_J$ at $\beta$, and then push $B_J$ across this disk. A schematic picture is shown in Figure \ref{figure, push beta across W}. 
\begin{figure}[htb!]
\begin{center}
\includegraphics[scale=0.55]{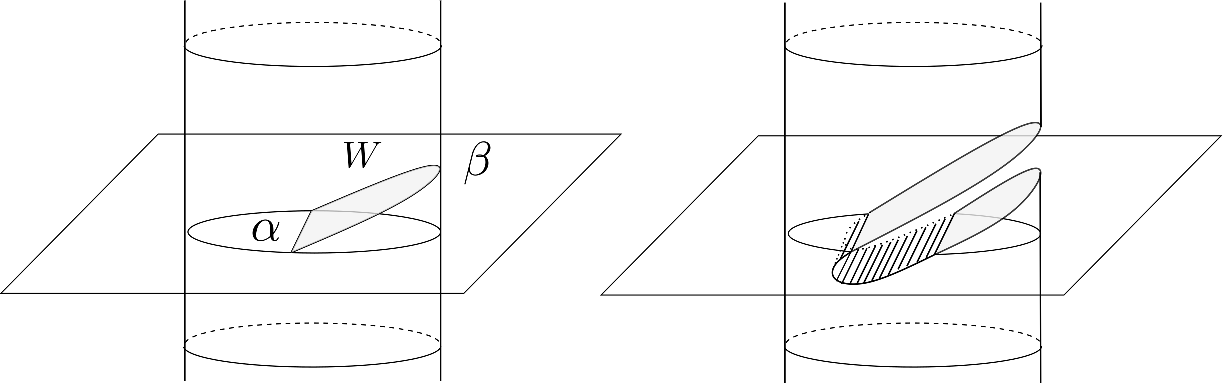}
\caption{A one dimension lower schematic picture for pushing $B_J$ across a disk to realize a stabilization of the cross-section.}\label{figure, push beta across W}
\end{center}
\end{figure}

We set this up more carefully. Let $\beta\times D^2\subset B_J$ be a normal bundle of $\beta$ such that $(\beta\times D^2)\cap (S_0\cup\cdots\cup S_n)=\{\beta(0),\beta(1)\}\times D^2$. Similarly, let $\alpha\times D^2\subset S^3$ be a normal bundle of $\alpha$ such that $\partial\alpha\times D^2$ is identified with $\partial \beta \times D^2$. We want an embedded disk $W$ such  that $W\cap (B_J\cup\alpha)=\beta\cup\alpha$, and $W$ admits a normal bundle $W\times D^2$ such that $\partial W\times D^2=(\alpha\times D^2)\cup (\beta\times D^2)$. 

To find such a $W$, we first construct a collar of $\partial W$ carefully. Let $B_J\times [-\epsilon,\epsilon]$ be a normal bundle of $B_J$ in $S^4$. Note $\alpha \cap (B_J\times[-\epsilon,\epsilon])=\{\alpha(0),\alpha(1)\}\times [0,\epsilon]$, for $(B_J\times[-\epsilon,\epsilon])\cap S^3$ is a normal bundle of the cross-section in $S^3$. Let $(\alpha\times D^2) \times [-t_0,t_0]$ be a normal bundle of $\alpha \times D^2$ in $S^4$ obtained by taking a product with a short interval in the 4-th dimension, where $t_0<1/3$. Then $((\alpha\times D^2) \times [-t_0,t_0])\cap (B_J\times [-\epsilon,\epsilon])=(\partial \alpha \times D^2) \times [0,\epsilon]\times [-t_0,t_0]$. (See Figure \ref{figure,thicken_alpha_and_B} and Figure \ref{figure,collar of boundary of W}.)
\begin{figure}[htb!]
\begin{center}
\includegraphics[scale=0.55]{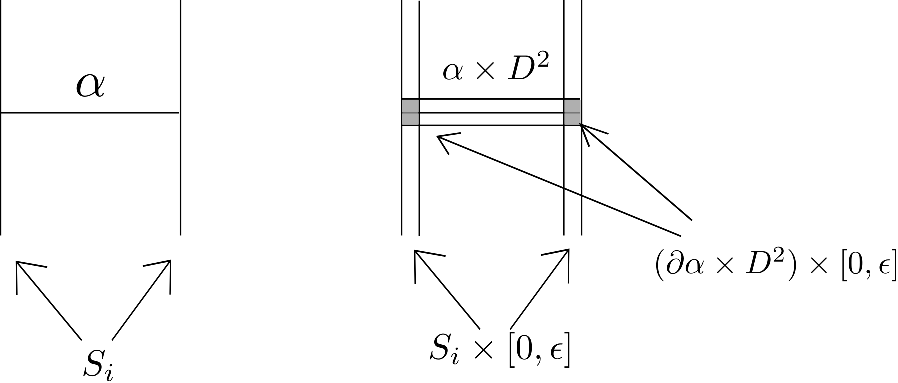}
\caption{A schematic picture for the intersection of $B_J\times[-\epsilon,\epsilon]$ and $(\alpha\times D^2)\times [-t_0,t_0]$ in $S^3$. The entire intersection can be viewed as the product of the intersection in $S^3$ with $[-t_0,t_0]$.}\label{figure,thicken_alpha_and_B}
\end{center}
\end{figure}

Let $\tilde{\alpha}=\alpha-(\{\alpha(0),\alpha(1)\}\times [0,\epsilon])$ and let $A=(\beta\times[0,\epsilon])\cup (\tilde{\alpha}\times[0,t_0])$. Then $A$ is an annulus with one of the boundary components being $\alpha\cup\beta$ (Figure \ref{figure,collar of boundary of W}). Clearly, $A$ admits a normal bundle that restricts to $(\alpha\cup\beta)\times D^2$ as desired. Denote the other component of $\partial A$ by $l$, note $l\subset \partial ((B_J\times [-\epsilon,\epsilon])\cup(\alpha\times D^2 \times [-t_0,t_0]))$. Note $(B_J\times [-\epsilon,\epsilon])\cup(\alpha\times D^2 \times [-t_0,t_0])$ is homeomophic to $S^1\times D^3$ and $l$ is isotopic to $S^1\times \{pt\}\subset S^1\times \partial D^3$. As any embedded $S^1$ is unknotted in $S^2\times D^2$, $l$ bounds a disk $D$ in the closed complement of $(B_J\times [-\epsilon,\epsilon])\cup(\alpha\times D^2 \times [-t_0,t_0])$. The desired disk is then $W=A\cup D$ and has a normal bundle extending the one on $A$ (\cite{MR1201584} Theorem 9.3A). 
\begin{figure}[htb!]
\begin{center}
\includegraphics[scale=0.55]{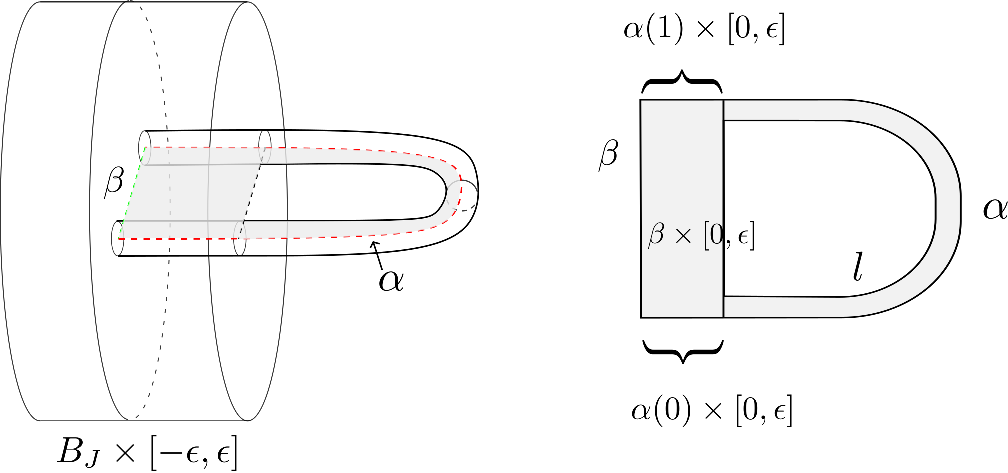}
\caption{Schematic picture of $(B_J\times[-\epsilon,\epsilon])\cup (\alpha\times D^2\times[-t_0,t_0])$ (left) and the annulus $A$ (right).}\label{figure,collar of boundary of W}
\end{center}
\end{figure}

Note by construction $A\cap S^3=\alpha\subset \partial W$. Since $W=A\cup D$, by transversality the interior of $W$ intersects $S^3$ at a collection of circles contained in $D$, which is a link $L\subset S^3$ in the complement of the $S_i's$. Modify $B_J$ by pushing $\beta\times D^2$ across $W\times D^2$. Equivalently, this is same as removing the interior of $\beta\times D^2$ and glue back $\partial (W\times D^2)\backslash (\beta\times D^2)$. Call the resulting 3-ball $\tilde{B}_J$, note $\tilde{B}_J\cap S^3=S_0\cup \cdots \cup S_n \cup (\alpha\times D^2) \cup \partial(nb(L))$. Finally, push $\alpha\times \text{Int}(D^2)$ slightly into the 4-th dimension to yield the desired ball $B_J'$ with a cross-section as stated in the lemma. 
\end{proof}

With Lemma \ref{lemma, stabilizing the cross-section} at hand, we can modify an arbitrary superslice ball to obtain favorable cross-sections. More concoretely, we have the following proposition.
\begin{prop}\label{Proposition, superslice ball cross-section arrangement}
For any knot $K$ there exists some Seifert surface $F_K$ such that
\begin{itemize}
\item[(1)] there exists a knot $J$ contained in $F_K$ satisfying the conclusion of Proposition \ref{feller's proposition}.  
\item[(2)] $J$ has a superslice ball $B_J$ such that $B_J\cap F_K=F_J$.
\end{itemize}
\end{prop}

\begin{proof}

Let $F_K'$ be an arbitrary Seifert surface for $K$. Apply Proposition \ref{feller's proposition} to get $F_K'=C_{J,K}\cup_J F_J'$, where $F_J'$ is a Seifert surface for $J$. According to Proposition \ref{proposition, knots with trivial alexander polynomial is trivial}, $J$ is superslice. Let $B_J'$ be a superslice ball for $J$ such that $B_J'\cap S^3= F_J''\cup S_1\cup \cdots \cup S_n$, where $F_J''$ is some Seifert Surface for $J$ and $S_1,\ldots ,S_n$ are closed orientable surfaces. Let $F_J$ be a common stabilization of $F_J'$ and $F_J''$. Our goal will be achieved by stabilizing both $F_K'$ and $B_J' \cap S^3$ properly. We divide the procedure into three steps.

\textbf{Step 1.} Stabilize $F_K'$ to get $F_K=C_{K,J}\cup F_J$. Note this is possible, for if an arc we use to stabilize $F_J'$ intersects $C_{K,J}$, then we can push the arc off $C_{K,J}$ along a path from the intersection point to the $K$-boundary. See Figure \ref{figure, pushing off intersection}. 
\begin{figure}[htb!]
\begin{center}
\includegraphics[scale=0.5]{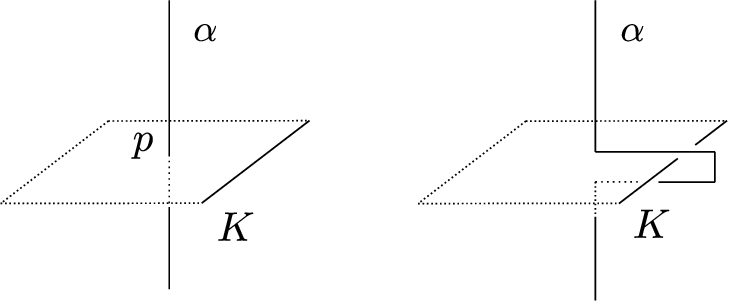}
\caption{Pushing an intersection off $C_{K,J}$.}\label{figure, pushing off intersection}
\end{center}
\end{figure}

\textbf{Step 2.} Isotope $B_J'$ to $B_J''$ so that the Seifert surface appeared in the cross-section $B_J''\cap S^3$ is $F_J$. We explain how to achieve this by assuming $F_J$ can be obtained from $F_J''$ by a single stabilization along an arc $\alpha$. The general case easily follows by repeating the argument. 

We induct on the number of intersection points of $\alpha \cap (S_1\cup\cdots\cup S_n)$. Denote this number by $2k$. Suppose there are no intersection points between $\alpha$ and the $S_i$'s, then Step 2 can be accomplished by applying Lemma \ref{lemma, stabilizing the cross-section}. Assume Step 2 can be accomplished when there are fewer than $2k$ intersection points. Now suppose there are $2k$ intersection points with $k>0$. Then the $\alpha$ arc is divided into subarcs by these points. There must be a subarc $\alpha'$ such that the end points of $\alpha'$ lie on some surface $S_i$ and the interior of $\alpha'$ does not intersect any of the surfaces. To see this, assume the arc $\alpha$ intersects surfaces $S_{m_1},\ldots, S_{m_l}$ for some $l\leq n$. There must be a surface $S_{m_i}$ which bounds a region that does not contain the end points of $\alpha$ and any other surface $S_{m_j}$, $j\neq i$. Note $\alpha$ must intersect such a surface consecutively at two points, and we take $\alpha'$ to be the subarc between two such intersection points. Apply Lemma \ref{lemma, stabilizing the cross-section} to $\alpha'$. Note $\alpha$ does not intersect the newly appeared tori $\partial(nb(L))$ as we may assume $\alpha\cap L=\emptyset$ by a general position argument, and we can arrange $\alpha$ to go through the newly appeared tunnel gained from the stabilization. Therefore, we have two fewer intersection points between $\alpha$ and the new cross-section, and the claim follows from the inductive hypothesis. 
See Figure \ref{figure, stabilizing S to create channels}.
\begin{figure}[htb!]
\begin{center}
\includegraphics[scale=0.55]{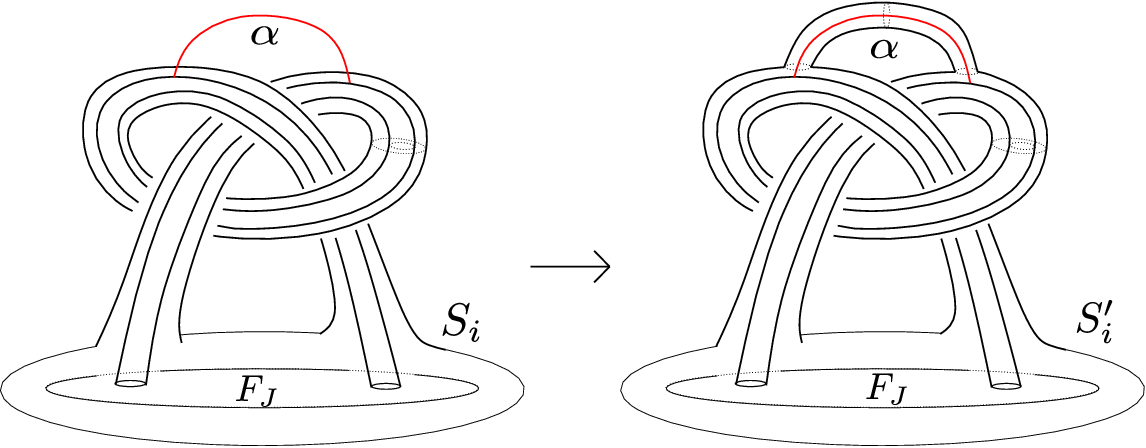}
\caption{Stabilizing $S_i$ to create tunnels for stabilizations of $F_J$.}\label{figure, stabilizing S to create channels}
\end{center}
\end{figure}

\textbf{Step 3.} Isotope $B_J''$ to $B_J$ so that the closed components in the cross-section do not intersect $C_{K,J}$. To see this, note to build $F_K$ from $F_J$ one needs to add a number of bands, and $C_{K,J}$ is isotopic to the union of a collar of $J$ and the bands. It is clear that if the core of a band does not intersect the $S_i$'s, then the whole band does not intersect the $S_i$'s. If not, suppose the core of the band does intersect the surfaces, then a stabilization arrangement as in Step 2 can be employed to remove the intersections. Therefore, after possibly stabilizing the $S_i$'s and introducing more tori, we may assume these closed surfaces do not intersect $C_{K,J}$. 
\end{proof}

Theorem \ref{superslcie upper bound} now follows readily from Proposition \ref{Proposition, superslice ball cross-section arrangement}.
\begin{proof}[Proof of Theorem \ref{superslcie upper bound}]
Given a knot $K$, let $F_K$ and $B_J$ be a Seifert surface for $K$ and a superslice ball for $J$ respectively as in Proposition \ref{Proposition, superslice ball cross-section arrangement}. Thicken $F_K$ slightly to $F_K\times [-\epsilon,\epsilon]$ using the 4-th dimension. By transversality we may assume $B_J \cap (F_K\times [-\epsilon,\epsilon])=F_J\times[-\epsilon,\epsilon]$ provided $\epsilon$ is small enough. Now construct a 3-dimensional handlebody $H=(F_K\times [-\epsilon,\epsilon])\cup B_J$. It is easier to see $H$ is a handlebody by viewing $H=(C_{K,J}\times [-\epsilon,\epsilon])\cup_{J\times[-\epsilon,\epsilon]}B_J$. Denote by $D_J$ the superslice disk that doubles to $\partial B_J$. Then $\partial(H)$ is the double of a surface isotopic to $C_{K,J}\cup D_J$ relative to $K$ in the 4-ball. Finally, note $2g(C_{K,J}\cup D_J)=\text{deg} (\Delta_K(t))$.
\end{proof}

As we mentioned in the introduction, this theorem has an immediate corollary which says that if a knot $K$ has $\deg(\Delta_K(t))=2$, then $g^s_{top}(K)=1$.

\begin{proof}[Proof of Corollary \ref{superslice genus for knots with alex of deg 1}]
It is understood $g^s_{top}(K)=0$ if and only if $\Delta_K(t)=1$ (\cite{LM15} \cite{GS75}). Therefore $\deg(\Delta_K(t))=2$ implies $g^s_{top}(K)>0$. In view of Theorem \ref{superslcie upper bound}, $g^s_{top}(K)\leq 1$.
Hence we have $g^s_{top}(K)=1$.
\end{proof}

\appendix
\section{}
In this appendix we prove Proposition \ref{technicalproposition}. We begin by computing the Casson-Gordon invariants of $L(9,4)$. The author found it convenient to use the strategy in \cite{Gil81} to describe the characters and compute the invariants.

Note that $L(9,4)$ admits a surgery diagram as shown in Figure \ref{figure, surgery diagram for L(9,4)}.
\begin{figure}[!ht]
\center
\includegraphics[scale=0.25]{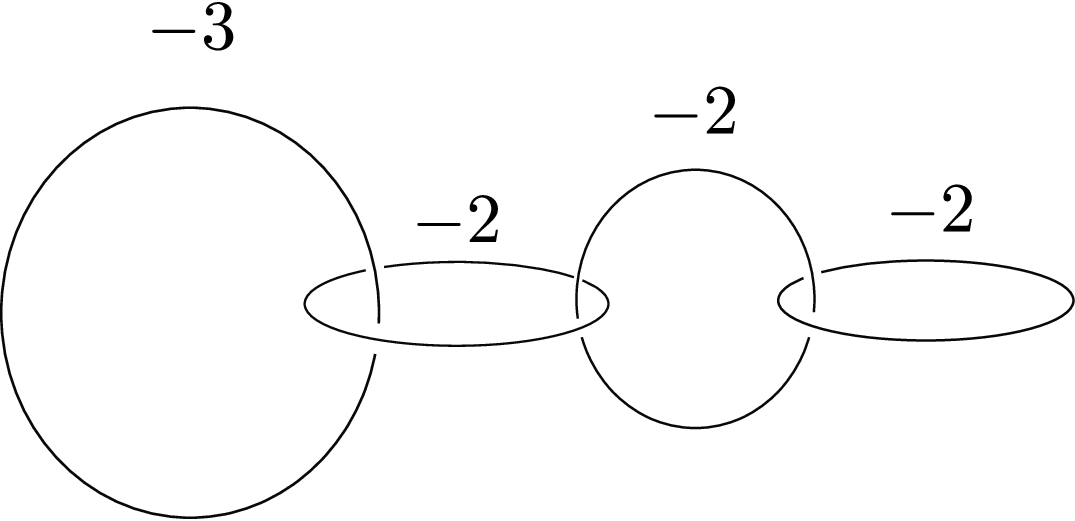}
\caption{A surgery diagram for $L(9,4)$.}
 \label{figure, surgery diagram for L(9,4)}
\end{figure}

We fix an isomorphism $H_1(L(9,4))\cong \mathbb{Z}_9$ by taking the homology class generated by the meridian  (with an arbitrarily chosen orientation) of the $-3$-framed unknot to be $1\in \mathbb{Z}_9$. Let $\chi_a: H_1(L(9,4)) \rightarrow \mathbb{Z}_9$ be the character sending $1$ to $a$. 
 
Applying a formula in \cite{Gil81} (on page 370) one gets $\sigma(L(9,4),\chi_a)$ as shown in the table:\\

\begin{center}
{\renewcommand{\arraystretch}{1.2}
\begin{tabular}{|c|c|c|c|c|c|c|c|c|c|}
\hline
a & 0 & 1 & 2 & 3 & 4 & 5 & 6 & 7 & 8 \\
\hline
$\sigma(L(9,4),\chi_a)$ & 0 & $\frac{5}{9}$ & $\frac{11}{9}$  & 1 & $-\frac{1}{9}$ & $-\frac{1}{9}$ & 1  & $\frac{11}{9}$ & $\frac{5}{9}$  \\
\hline
\end{tabular}
}
\end{center}

We need the following lemma before proving Proposition \ref{technicalproposition}.
\begin{lem}\label{basislemma}
Let $m$, $n$ be nonnegative integers and let $s: H_1(\#_nL(9,4))\rightarrow \oplus_m \mathbb{Z}_9$ be a surjective homomorphism. Identify $H_1(\#_nL(9,4))\cong \oplus_n\mathbb{Z}_9$ using the isomorphism described above.  Then one can choose a basis for $\oplus_m\mathbb{Z}_9$ and reorder the basis of $\oplus_n\mathbb{Z}_9$ if necessary, so that in terms of these two bases, $s$ can be represented by a matrix of the form
$$\begin{bmatrix}
  1 & \cdots & 0 & a^{m+1}_1 & \cdots & a^n_1  \\
    \vdots  & \ddots & \vdots & \vdots & \ddots &\vdots\\
       0    & \cdots  & 1 & a^{m+1}_m & \cdots & a^n_m
\end{bmatrix}_{\textstyle \raisebox{2pt}{.}}$$
Here the entries of the matrix are counted mod $9$.
\end{lem}
\begin{proof}
Start with an arbitrary basis for $\oplus_m\mathbb{Z}_9$. Write $s$ as a matrix
$$\begin{bmatrix}
  a^1_1 & \cdots & a^m_1 & a^{m+1}_1 & \cdots & a^n_1  \\
    \vdots  & \ddots & \vdots & \vdots & \ddots &\vdots\\
     a^1_m    & \cdots  & a^m_m & a^{m+1}_m & \cdots & a^n_m
\end{bmatrix}_{\textstyle \raisebox{2pt}{.}}$$
Since $s$ is surjective, then the first row represents a surjective map $\oplus_n\mathbb{Z}_9$ to $\mathbb{Z}_9$. Therefore, there must exists $a^i_1$ for some $i$ such that $3\nmid a^i_1$. So $a^i_1$ is a generator in $\mathbb{Z}_9$, and we may alter the first basis element of $\oplus_m\mathbb{Z}_9$ by a proper multiplication, so that under the new basis, $a^i_1$ becomes $1$. Reorder the basis for $H_1(\#_nL(9,4))$ if necessary to achieve $a^1_1=1$. 

Then under the new bases, $s$ has the following matrix representation (by abusing notation we still denote the undetermined values by $a^i_j$):
$$\begin{bmatrix}
  1 & \cdots & a^m_1 & a^{m+1}_1 & \cdots & a^n_1  \\
    \vdots  & \ddots & \vdots & \vdots & \ddots &\vdots\\
     a^1_m    & \cdots  & a^m_m & a^{m+1}_m & \cdots & a^n_m
\end{bmatrix}_{\textstyle \raisebox{2pt}{.}}$$
Write for now the basis for $\oplus_m\mathbb{Z}_9$ as $e_1,...,e_m$. Make a change of basis: $e_1'=e_1+\Sigma_{j=2}^m a^1_j e_j$, $e_i'=e_i$ for $i=2,3,...,m$.  Then under the new basis $\{e_i'\}$,  (abusing notation again) $s$ has the following matrix representation:
$$\begin{bmatrix}
  1 & a^2_1&\cdots & a^m_1 & a^{m+1}_1 & \cdots & a^n_1  \\
   0 & a^2_2  & \cdots  & a^m_2 & a^{m+1}_m & \cdots & a^n_m\\
    \vdots &\vdots & \ddots & \vdots & \vdots & \ddots &\vdots\\
     0 &a^2_m   & \cdots  & a^m_m & a^{m+1}_m & \cdots & a^n_m
\end{bmatrix}_{\textstyle \raisebox{2pt}{.}}$$

Repeating this process for each row leads to a matrix of the desired form.
\end{proof}

We are ready to prove Proposition \ref{technicalproposition}.

\begin{proof}[Proof of Proposition \ref{technicalproposition}]
By Lemma \ref{basislemma}, we may assume $s$ is of the form 
$$\begin{bmatrix}
  1 & \cdots & 0 & a^{m+1}_1 & \cdots & a^n_1  \\
    \vdots  & \ddots & \vdots & \vdots & \ddots &\vdots\\
       0    & \cdots  & 1 & a^{m+1}_m & \cdots & a^n_m
\end{bmatrix}_{\textstyle \raisebox{2pt}{.}}$$
We want to choose a map $j:\oplus_m\mathbb{Z}_9\rightarrow \mathbb{Z}_9$  so that the corresponding Casson-Gordon invariant is big enough. Write $j$ as 
$$\begin{bmatrix}
  j_1 & j_2 &  \dots & j_m
\end{bmatrix}_{\textstyle \raisebox{2pt}{.}}$$

Using the additivity of Casson-Gordon invariants (e.g. \cite{Jia81}), we have
\begin{displaymath}
\begin{aligned}
\sigma(\#_nL(9,4),&j\circ s)=\sigma(L(9,4), \chi_{j_1})+...+\sigma(L(9,4),\chi_{j_m})\\
&+\sigma(L(9,4), \chi_{j_1a^{m+1}_1+...+j_ma^{m+1}_m})+...+\sigma(L(9,4), \chi_{j_1a^{n}_1+...+j_m a^{n}_m}).
\end{aligned}
\end{displaymath}

We encode the last $n-m$ terms of the above equation into a vector $$H(j_1,j_2,...,j_m)=[\sigma(L(9,4), \chi_{j_1a^{m+1}_1+...+j_ma^{m+1}_m}),...,\sigma(L(9,4), \chi_{j_1a^{n}_1+...+j_ma^{n}_m})].$$

Consider $H(2,2,...,2)$. If $H(2,2,...,2)$ has less than $m$ entries of value $-\frac{1}{9}$, let $j=[2,2,...,2]$.  Since all other values in the table are nonnegative, we have 
\begin{displaymath}
\begin{aligned}
\sigma(\#_nL(9,4),j\circ s)=&\sigma(L(9,4), \chi_2)+...+\sigma(L(9,4),\chi_2)\\
+\sigma(L(9,4), &\chi_{2a^{m+1}_1+...+2a^{m+1}_m})+...+\sigma(L(9,4), \chi_{2a^{n}_1+...+2 a^{n}_m})\\
&\geq \frac{11}{9}m-\frac{1}{9}m\\
&=\frac{10}{9}m.
\end{aligned}
\end{displaymath}

If $H(2,2,...,2)$ has at least $m$ entries of value $-\frac{1}{9}$, then choose $j=[6,6,...,6]$. Using the table one sees the entries of value $-\frac{1}{9}$ in $H(2,2,...,2)$ all become $1$ in $H(6,6...,6)$. Also note that there will not be any negative values in $H(6,6,...,6)$ since $3| (6a^{i}_1+...+6 a^{i}_m)$ for any $i$ and such character always corresponds to value $1$ or $0$. Therefore, in this case
\begin{displaymath}
\begin{aligned}
\sigma(\#_nL(9,4),j\circ s)=&\sigma(L(9,4), \chi_6)+...+\sigma(L(9,4),\chi_6)\\
+\sigma(L(9,4), &\chi_{6a^{m+1}_1+...+6a^{m+1}_m})+...+\sigma(L(9,4), \chi_{6a^{n}_1+...+6a^{n}_m})\\
&\geq m+m\\
&\geq \frac{10}{9}m.
\end{aligned}
\end{displaymath}
\end{proof}

\bibliographystyle{abbrv}
\bibliography{doubref}
\end{document}